\documentclass[reqno,11pt,a4paper]{amsart}

\usepackage{amsfonts,amsmath,amssymb,amsthm,color}
\usepackage{mathrsfs}
\usepackage[USenglish]{babel}
\usepackage{esint}
\usepackage{graphicx}
\usepackage[colorlinks=true]{hyperref}
\usepackage[latin1]{inputenc}
\usepackage[hmargin=3cm,vmargin=2.9cm]{geometry}

\hypersetup{colorlinks,citecolor=red,linkcolor=blue,urlcolor=black}

\newcommand\e\varepsilon

\newcommand\R{\mathbb R}

\newcommand\D{\mathbb D}
\renewcommand{\(}{\left(}
\renewcommand{\)}{\right)}

\newtheorem{theorem}{Theorem}[section]
\newtheorem{lemma}[theorem]{Lemma}

\newtheorem{proposition}[theorem]{Proposition}

\newtheorem{corollary}[theorem]{Corollary}

\allowdisplaybreaks

\everymath{\displaystyle}

\begin{document}

\title[Blow-up for a non-local Liouville equation]{A blow-up phenomenon for a non-local Liouville-type equation}

\author{Luca Battaglia}
\address[L. Battaglia]{Dipartimento di Matematica e Fisica, Universit\`a degli Studi Roma Tre, Largo S. Leonardo Murialdo, 00146 Roma,Italy}
\email{lbattaglia@mat.uniroma3.it}

\author{Mar\'ia Medina}
\address[Mar\'ia Medina]{Departamento de Matem\'aticas,
Universidad Aut\'onoma de Madrid,
Ciudad Universitaria de Cantoblanco,
28049 Madrid, Spain}
\email{maria.medina@uam.es}

\author{Angela Pistoia}
\address[A. Pistoia]{Dipartimento SBAI, ``Sapienza" Universit\`a di Roma, via Antonio Scarpa 16, 00161 Roma,Italy}
\email{angela.pistoia@uniroma1.it}

\begin{abstract}
We consider the non-local Liouville equation
$$(-\Delta)^\frac12u=h_\e e^u-1\ \hbox{in}\ \mathbb S^1,$$
corresponding to the prescription of the geodesic curvature on the circle. We build a family of solutions which blow up, when $h_\e$ approaches a function $h$ as $\e\to0$, at a critical point of the harmonic extension of $h$ provided some generic assumptions are satisfied.

\end{abstract}

\date\today
\thanks{M. Medina was partially supported by the European Union's Horizon 2020 research and innovation programme under the Marie Sklodowska-Curie grant agreement N 754446 and UGR Research and Knowledge Transfer Fund - Athenea3i.\\ 
A. Pistoia was partially supported by Fondi di Ateneo ``Sapienza" Universit\`a di Roma (Italy).\\
M. Medina wants to acknowledge the hospitality of Universit\`a La Sapienza di Roma, where this work was carried out during a long visit in the academic year 2019--2020.}

\maketitle

\section{Introduction}

The classical Nirenberg problem consists in finding positive functions $h$ on the standard sphere $(\mathbb S^n,g_0)$ for which there exists a metric $g$ conformally equivalent to $g_0$ whose scalar curvature is equal to $h.$ In dimension $n=2$, the Nirenberg problem asks what functions can be the gaussian curvature of a conformal metric on $\mathbb S^2.$ It can also be rephrased in terms of solutions of a partial differential equation on the sphere. More precisely, one looks for functions $h$ on $\mathbb S^2$ for which there exists a solution $u:\mathbb S^2\to\mathbb R$ of the Liouville equation
\begin{equation}\label{e1}-\Delta_{\mathbb S^2}u=he^{2u}-1\quad \hbox{in}\; \mathbb S^2.
\end{equation}
Indeed a straightforward computation shows that the gaussian curvature of the conformal metric $g=e^{2u}g_0$ is nothing but the prescribed function $h.$
Recently, Da Lio, Martinazzi and Rivi\'ere in \cite{dmr} investigated the case $n=1$. They parametrize a planar Jordan curve (i.e. a continuous closed and simple curve) through the trace of the Riemann mapping between the disk $\D$ and the simply connected domain enclosed by the curve and find an equation similar to \eqref{e1}
\begin{equation}\label{e2}
(-\Delta)^\frac12u=he^u-1\quad\hbox{in}\;\mathbb S^1,
\end{equation}
whose solutions give the curvature density $he^u d\theta$ of the curve in this parametrization. Here $(-\Delta)^\frac12$ is the $\frac12$-fractional Laplacian in $\mathbb S^1$, i.e.
$$(-\Delta)^\frac12u(z)=\frac1\pi\hbox{p.v.}\int_{\mathbb S^1}\frac{u(z)-u(w)}{|z-w|^2}dw.$$
Problem \eqref{e2} is equivalent, up to a constant factor 2, to 
\begin{equation}\label{e3}\left\{\begin{array}{ll}-\Delta u=0&\hbox{in}\;\D,\\\partial_\nu u+2=2he^\frac u2&\hbox{on}\;\partial\D=\mathbb S^1,\end{array}\right.
\end{equation}
where $\nu$ is the outward-pointing normal derivative at the boundary. Problem \eqref{e3} corresponds to the geometric problem of finding a flat metric $g$ on the disk $(\D,g_0)$ such that $h$ is the geodesic curvature of $\mathbb S^1$ with respect to the metric $g=e^{2u}g_0$ which is pointwise conformal to $g_0$.\\

Necessary conditions on $h$ to solve \eqref{e2} or \eqref{e3} are easily obtained. Indeed integrating \eqref{e3} we get
\begin{equation}\label{n1}\int\limits_{\mathbb S^1}he^\frac u2d\sigma_{g_0}=2\pi,\hbox{which implies}\ \max\limits_{\mathbb S^1}h>0.\end{equation}
As far as we know there are few results about existence and multiplicity of solutions of \eqref{e2} or \eqref{e3}. The first one seems due to Chang and Liu \cite{cl}, who proved the existence of a solution to \eqref{e3} provided $h$ is positive, has only isolated critical points, $\widehat h'(z)\ne0$ whenever $h(z)=0$ where $\widehat h$ denotes the complex conjugate function of $h$, and a further relation between local maxima and local minima of $h$ holds true. Later, Liu and Huang in \cite{lh} considered the case where $h$ possesses symmetry and they found a solution to \eqref{e3} if $h$ has a minimum point $z_0$ which satisfies $\partial_\nu H(z_0)>0$, where $H$ is the harmonic extension of $h$ to all of $\D,$ i.e.

\begin{equation}\label{he}\left\{
\begin{array}{ll}-\Delta H=0&\hbox{in}\;\D,\\H=h&\hbox{on}\;\mathbb S^1.\end{array}\right.\end{equation}

Finally, Zhang \cite{z} employed a negative gradient flow method to build a solution to \eqref{e3} when the necessary condition in \eqref{n1} is satisfied. The existence issue is strictly related to the study of the blow-up phenomenon. Recently Jevnikar, L\'opez-Soriano, Medina and Ruiz in \cite{jlmr} proved that if $u_n$ is a sequence of blow-up solutions of \eqref{e2} or \eqref{e3} with $h$ replaced by $h_n,$ $h_n$ is uniformly bounded in $C^2\(\mathbb S^1\)$ and $e^{u_n}$ is uniformly bounded in $L^1\(\mathbb S^1\)$, then $u_n$ blows-up at a unique point $p\in\mathbb S^1$ such that $h(p)>0$ and $\nabla H(p)=0$ where $H$ is the harmonic extension of $h$ (see \eqref{he}). See also \cite{gl} for results concerning the blow-up analysis for problem \eqref{e2}.\\

The goal of this article is to provide the first example of blow-up phenomenon for the problems \eqref{e2} and \eqref{e3}. To do so we will construct a family of solutions to some approximated problems, with $h$ perturbed as a function $h_\e\to h$ in $C^1\(\mathbb S^1\)$ as $\e\to0$, which concentrate at one point as $\e\to0$. In particular, our result also gives the first multiplicity result to \eqref{e2}/\eqref{e3}. More precisely, consider 
\begin{equation}\label{heps}
h_\e(z):=h(z)+\e k(z),
\end{equation}
with $\e>0$ a small parameter, $h\in C^{2,\alpha}\(\mathbb S^1\)$ for some $\alpha>0$ and $k\in C^1\(\mathbb S^1\)$, and the problem
\begin{equation}\label{pf}
(-\Delta)^\frac12u=h_\e(z)e^u-1\quad\hbox{in}\;\mathbb S^1.
\end{equation}
or equivalently
\begin{equation}\label{p}\left\{\begin{array}{ll}-\Delta u=0&\hbox{in}\;\D,\\\partial_\nu u+2=2 h_\e(z)^\frac u2&\hbox{on}\;\mathbb S^1.\end{array}\right.
\end{equation}
We will build a family of solutions to \eqref{pf} or \eqref{p} which blow-up at any point $\xi_0\in\mathbb S^1$ around which $h$ satisfies suitable conditions. For sake of simplicity, we will assume that  $\xi_0=1$. According to the blow-up analysis performed in \cite{jlmr} we need to assume
\begin{equation}\label{h1}
h(1)>0\quad \hbox{and}\quad h'(1)=(-\Delta)^\frac12h(1)=0,
\end{equation}
where $h'$ stands for the tangential derivative of $h$. We also require the following non-degeneracy condition at the point $1$,
\begin{equation}\label{nondeg}
\(h''(1)-\frac2{\pi^2}Q(h){1\over h(1)}\)h''(1)+\((-\Delta)^\frac12h'(1)\)^2\ne0,
\end{equation}
where
\begin{equation}\label{q}
Q(h):=\int_{\mathbb S^1\times\mathbb S^1}\log\frac1{|z-w|}\frac{h(z)-h(1)}{|z-1|^2}\frac{h(w)-h(1)}{|w-1|^2}dwdz.
\end{equation}
This condition is not restrictive since if the equality holds we can replace $h(z)$ with $h(z)+c$ being $c$ a small constant and \eqref{nondeg} holds provided $c$ is small enough. We need $h''$ to be H\"older continuous in order for $Q(h)$ to appear in the main term in the energy expansion (see Propositions \ref{E1}, \ref{E2}).

A similar argument shows that, for the sake of simplicity, we can also assume the perturbative term $k$ to vanish at the point $1,$ i.e.
$k(1)=0.$
Finally, we assume the transversality condition
\begin{equation}\label{cond}
h''(1)(-\Delta)^\frac12k(1)-k'(1)(-\Delta)^\frac12h'(1)\ne0.\end{equation}
This is quite natural, since a simple computation shows that \eqref{cond} is equivalent to requiring that the vectors $\partial_{\xi_2}\nabla_\xi\(H+\e K\)|_{\xi=1,\e=0}$ and $\partial_\e\nabla_\xi\(H+\e K\)|_{\xi=1,\e=0}$ are not parallel. Here $H$ and $K$ are the harmonic extensions of $h$ and $k$, respectively.\\

Our main result reads as

\begin{theorem}\label{main}
Assume $h\in C^{2,\alpha}\(\mathbb S^1\)$ and $k\in C^1\(\mathbb S^1\)$. Suppose that $\xi_0=1\in\mathbb S^1$ satisfies \eqref{h1} and \eqref{nondeg}. If \eqref{cond} holds true, then there exists $\e_0>0$ such that, for every $\e\in(0,\e_0)$ or for every $\e\in(-\e_0,0)$ there exists a solution $u_\e$ of 
$$(-\Delta)^\frac12u=h_\e(z) e^u-1\quad\hbox{in}\;\mathbb S^1,$$
blowing-up at $\xi_0=1$ as $\e\to0$, with $h_\varepsilon$ defined at \eqref{heps}.

Furthermore, there exist $\delta_\e>0$ and $\xi_\e\in\mathbb S^1$ with
$\delta_\e=O(\e)$ and $\xi_\e=1+O(\e)$
such that
$$u_\e\(f_{\delta_\e,\xi_\e}(z)\)+\log\left|f'_{\delta_\e,\xi_\e}(z)\right|+\log h(1)=O\(\e\)\ \hbox{in}\ L^p\(\mathbb S^1\)\quad\forall p\in [1,+\infty),$$
where $f_{\delta,\xi}$ is the conformal map 
\begin{equation}\label{confo}f=f_{\delta,\xi}(z):=\frac{z+(1-\delta)\xi}{1+(1-\delta)\overline\xi z}.\end{equation}
\end{theorem}

This theorem completes our previous work \cite{bmp}, where we studied the problem of prescribing the gaussian and geodesic curvatures for a conformal metric on the unit disk, which turns out to be equivalent to solve the problem
\begin{equation}\label{kh}
\begin{cases}-\Delta u=2K(z)e^u&\hbox{in}\;\D,\\
\partial_\nu u+2=2h(z)e^\frac u2&\hbox{on}\;\mathbb S^1,\end{cases}
\end{equation}
where $K,h$ are the prescribed curvatures. The reader can find an exhaustive list of references concerning this problem in \cite{bmp}. In particular, there we built a family of conformal metrics with curvatures $K_\e,h_\e$ converging to $K,h$ respectively as $\e$ goes to $0$, which blows up at one boundary point under some generic assumptions. The strategy we follow in the present paper is similar, but it requires some careful estimates of the error term. 

Briefly, we consider the conformal map given in \eqref{confo} with $\delta=\delta_\e\to_{\e\to0}0$ and $\xi=\xi_\e=e^{\imath\eta_\e}\to_{\e\to0}1\in\mathbb S^1$ so that $(1-\delta)\xi\in\D$. 
Thus, letting $v(z):=u(f(z))+2\log|f'(z)|,$ we rewrite problem \eqref{pf} as
\begin{equation}\label{probv}(-\Delta)^\frac12v=h_\e(f(z))e^v-1\quad\hbox{in}\;\mathbb S^1,\end{equation}
and using a Ljapunov-Schmidt procedure we find a solution of \eqref{probv} as
\begin{equation}\label{vwtau}
v(z)\sim V(z)+W(z)+\tau,
\end{equation}
where the first order term is just a constant
\begin{equation}\label{Vbox}\boxed{V(z):=V_\xi(z)\equiv-\log h(\xi),}\end{equation}
solving
\begin{equation}\label{V}(-\Delta)^\frac12V=h(\xi)e^V-1\quad\hbox{on}\;\mathbb S^1.
\end{equation}
Non-trivial solutions to \eqref{V} have been classified by Ou \cite{ou} and Zhang \cite{zhang}.
The second order term, which is the key of the ansatz, is defined as

\begin{equation}\label{Wbox}
\boxed{W(z):=W_\xi(z)=\frac1\pi\int\limits_{\mathbb S^1}\log\frac1{|z-w|}(h(f(w))-h(\xi))e^Vdw,}
\end{equation}
solving
$$(-\Delta)^\frac12W=(h(f(z))-h(\xi))e^V-\frac1{2\pi}\int\limits_{\mathbb S^1}(h(f(w))-h(\xi))e^Vdw\quad\hbox{in}\;\mathbb S^1;$$
finally, $\tau$ is a small constant.\\

Regarding the case considered in \cite{bmp}, the construction of the solutions needs to be more careful. Actually, the non-degeneracy condition corresponding to \eqref{nondeg} for problem \eqref{kh} involves the second-order derivatives of the interior curvature $K$ (see (7) and (8) in \cite{bmp}), excluding the case $K\equiv0$ considered here.

Quite surprisingly, assumptions in Theorem \ref{main} look simpler and more natural compared to \cite[Theorem 1.1]{bmp}.
In particular, the non-degeneracy condition \eqref{nondeg} involves not only the second derivatives of $h$, but also the quadratic form $Q(h)$ defined in \eqref{q}, which has an interesting interpretation. As a first thing, we notice that the function
$$\widehat h(z):=\frac{h(z)-h(1)}{|z-1|^2}$$
is bounded and has zero average, due to \eqref{h1}; therefore, we may consider the (zero-average) solution $\widetilde h$ to
$$(-\Delta)^\frac12\widetilde h=\widehat h\quad\hbox{in}\;\mathbb S^1.$$
Using Green's representation, we can write
$$Q(h)=\int_{\mathbb S^1}\widehat h(z)\(\int_{\mathbb S^1}\log\frac1{|z-w|}\widehat h(w)dw\)\,dz=\int_{\mathbb S^1}\widehat h(z)\widetilde h(z)=\int_{\mathbb S^1}(-\Delta)^\frac12\widetilde h(z)\widetilde h(z).$$
In particular, since $(-\Delta)^\frac12$ is a positive definite operator, then $Q(h)$ is a positive definite quadratic form and, since $h\not\equiv0$, we get $Q(h)>0$.\\

The plan of the paper is as follows: in Section 2 we provide crucial estimates for the main term $W$ in the ansatz; in Section 3 we develop the linear theory and solve the auxiliary fixed-point problem for $\phi$; in Section 4 we evaluate the projections on the kernels of the linearized operator and we conclude the proof of Theorem \ref{main}.

\medskip

\section{Ansatz and error estimates}

We look for a solution of \eqref{probv} as
$ v(z)=V(z)+W(z)+\tau+\phi(z), $
where $V$ and $W$ are defined in \eqref{Vbox} and \eqref{Wbox}, $\tau=\tau_\e\underset{\e\to0}\to0$ is a constant and $\phi(z)=\phi_{\xi,\delta,\tau}(z)$ is a small function to be found. 

In fact, we want to find $\xi,\delta,\tau$ such that $\phi$ solves
$$(-\Delta)^\frac12(V+W+\tau+\phi)=h_\e(f(z))e^{V+W+\tau+\phi}-1\quad\hbox{in}\;\mathbb S^1$$
that is
\begin{eqnarray*}
(-\Delta)^\frac12\phi-h(\xi)e^V\phi&=&\(h_\e(f(z))e^{W+\tau}-h(f(z))\)e^V+\frac1{2\pi}\int\limits_{\mathbb S^1}(h(f(w))-h(\xi))e^Vdw\\
&+&\(h_\e(f(z))e^{W+\tau}-h(\xi)\)e^V\phi\\
&+&h_\e(f(z))e^{V+W+\tau}\(e^\phi-1-\phi\)\quad\hbox{in}\;\mathbb S^1.
\end{eqnarray*}
This can be rewritten as
\begin{equation}\label{eqphi}
\mathcal L_0\phi=\mathcal E+\mathcal L\phi+\mathcal N(\phi)\quad\hbox{in}\;\mathbb S^1,
\end{equation}
with
\begin{eqnarray}
\nonumber\mathcal L_0\phi&:=&(-\Delta)^\frac12\phi-h(\xi)e^V\phi,\\
\label{ebdr}\mathcal E&:=&\(h_\e(f(z))e^{W+\tau}-h(f(z))\)e^V+\frac1{2\pi}\int\limits_{\mathbb S^1}(h(f(w))-h(\xi))e^Vdw,\\
\nonumber\mathcal L\phi&:=&\(h_\e(f(z))e^{W+\tau}-h(\xi)\)e^V\phi,\\
\nonumber\mathcal N(\phi)&:=&h_\e(f(z))e^{V+W+\tau}\(e^\phi-1-\phi\).
\end{eqnarray}

The following two auxiliary results will be useful along the paper, and can be found at \cite[Proposition 7.1]{bmp} and \cite[Proposition 7.2]{bmp} respectively.

\begin{proposition}\label{f}
Let $\xi\in\mathbb S^1$. For any $z\in\mathbb S^1$ one has
$$|f(z)-\xi|=O\(\frac\delta{\delta+|z+\xi|}\),$$
and in particular
$$\|f(z)-\xi\|_{L^p}=\begin{cases}O\(\delta\log\frac1\delta\)&p=1,\\O\(\delta^\frac1p\)&p>1.\end{cases}$$
Moreover, if $z\in\mathbb S^1$ and $h\in C^2\(\mathbb S^1\)$, then
$$h(f(z))-h(\xi)=\delta h'(\xi)\Theta(z)+O\(\frac{\delta^2}{(\delta+|z+\xi|)^2}\),$$
with
$$\Theta(z)=\Theta_{\delta,\xi}(z):=\frac{2\left\langle z,\xi^\perp\right\rangle}{1+(1-\delta)^2+2(1-\delta)\langle z,\xi\rangle}.$$
\end{proposition}

\begin{proposition}\label{intFracLap}
Given $\xi\in\mathbb S^1$,
$$\int\limits_{\mathbb S^1}(h(f(z))-h(\xi))dz=-2\pi\delta(-\Delta)^\frac12h(\xi)+O\(\delta^2\).$$
\end{proposition}

We can now give estimates on the correction term $W$ given in \eqref{Wbox}. 

\begin{lemma}\label{estW}The function $W$ satisfies 
\begin{equation}\label{w}
W(z)=O\(\delta|\eta|(1+\log|z+\xi|)+\frac{\delta^2}{\delta+|z+\xi|}\(1+\left|\log\frac{|z+\xi|}\delta\right|\)\),
\end{equation}
and
$$\|W\|_{L^p}=O\(\delta^{1+\frac1p}+\delta|\eta|\),\quad\left\|e^W\right\|_{L^p}=O(1),$$
for every $p\in[1,+\infty)$.
\end{lemma}

\begin{proof}
Estimate \eqref{w} follows by \cite[Lemma 2.1]{bmp} noticing that, as a consequence of \eqref{h1},
$$|h'(\xi)|+\left|\(-\Delta\)^\frac12h(\xi)\right|=O(|\eta|).$$
The $L^p$ estimates are straightforward.
\end{proof}

\begin{proposition}The correction term $W$ satisfies 
\begin{eqnarray}
\label{wz}W(z)&=&\frac2{\pi h(\xi)}\delta\int_{\mathbb S^1}\log\frac{|w-\xi|}{|f(z)-w|}\frac{h(w)-h(\xi)-h'(\xi)\left\langle w,\xi^\perp\right\rangle}{|w-\xi|^2}dw\\
\nonumber&+&O\(\delta\log\frac1\delta(\delta+|\eta|)\(1+\log^2\frac1{\delta+|z+\xi|}\)\).
\end{eqnarray}
\end{proposition}

\begin{proof}
We split the function into three parts,
\begin{eqnarray*}
W(z)&=&\underbrace{-\frac1{\pi h(\xi)}\log|(1-\delta)z+\xi|\int_{\mathbb S^1}(h(f(w))-h(\xi))dw}_{=:W_1(z)}\\
&+&\underbrace{\frac1{\pi h(\xi)}\int_{\mathbb S^1}\log\frac{|(1-\delta)z+\xi|}{|z-w|}\(h'(\xi)\frac{(2-\delta)\delta\left\langle w,\xi^\perp\right\rangle}{\delta^2+(1-\delta)|w+\xi|^2}\)dw}_{=:W_2(z)}\\
&+&\underbrace{\frac1{\pi h(\xi)}\int_{\mathbb S^1}\log\frac{|(1-\delta)z+\xi|}{|z-w|}\(h(f(w))-h(\xi)-h'(\xi)\frac{(2-\delta)\delta\left\langle w,\xi^\perp\right\rangle}{\delta^2+(1-\delta)|w+\xi|^2}\)dw}_{=:W_3(z)}.
\end{eqnarray*}
By Proposition \ref{intFracLap} the first one can be estimated as
\begin{eqnarray*}
W_1(z)&=&-\frac1{\pi h(\xi)}\log\sqrt{\delta^2+(1-\delta)|z+\xi|^2}\(\delta\(-\Delta\)^\frac12h(\xi)+O\(\delta^2\)\)\\
&=&O\(\(1+\log\frac1{\delta+|z+\xi|}\)\delta(\delta+|\eta|)\).
\end{eqnarray*}
Furthermore, by Green's representation formula, 
\begin{eqnarray*}
W_2(z)&=&\frac{\delta(2-\delta)h'(\xi)}{\pi h(\xi)}\int_{\mathbb S^1}\log|z-w|\frac{\left\langle w,\xi^\perp\right\rangle}{\delta^2+(1-\delta)|w+\xi|^2}dw\\
&=&O(\delta|\eta|)\(-\frac{2\pi}{1-\delta}\arctan\frac{(1-\delta)\left\langle w,\xi^\perp\right\rangle}{1+(1-\delta)\langle z,\xi\rangle}\)\\
&=&O\(\delta|\eta|\),
\end{eqnarray*}
by noticing that
$$k_2=-\frac{2\pi}{1-\delta}\arctan\frac{(1-\delta)\left\langle z,\xi^\perp\right\rangle}{1+(1-\delta)\langle z,\xi\rangle}$$
solves
$$(-\Delta)^\frac12k_2=-\frac{2\pi\left\langle w,\xi^\perp\right\rangle}{\delta^2+(1-\delta)|w+\xi|^2}\quad\hbox{in}\;\mathbb S^1.$$
To estimate $W_3$ we write
$$\frac{(1-\delta)z+\xi}{z-w}=\frac{(1-\delta)f(w)-\xi}{f(w)-f(z)},$$
and thus
\begin{equation*}\begin{split}
&W_3(z)=\frac1{\pi h(\xi)}\int_{\mathbb S^1}\log\frac{|(1-\delta)v-\xi|}{|v-f(z)|}\(h(v)-h(\xi)-h'(\xi)\left\langle v,\xi^\perp\right\rangle\)\frac{\delta(2-\delta)}{\delta^2+(1-\delta)|v-\xi|^2}dv\\
&=\frac2{\pi h(\xi)}\delta\int_{\mathbb S^1}\log\frac{|v-\xi|}{|v-f(z)|}\frac{h(v)-h(\xi)-h'(\xi)\left\langle v,\xi^\perp\right\rangle}{|v-\xi|^2}dv\\
&\;+\underbrace{\frac1{\pi h(\xi)}\delta\int_{\mathbb S^1}\log\frac{|v-f(z)|}{|v-\xi|}\(h(v)-h(\xi)-h'(\xi)\left\langle v,\xi^\perp\right\rangle\)\(\frac2{|v-\xi|^2}-\frac{2-\delta}{\delta^2+(1-\delta)|v-\xi|^2}\)dv}_{:=W_{3,1}(z)}\\
&\;+\underbrace{\frac1{\pi h(\xi)}\int_{\mathbb S^1}\log\frac{|(1-\delta)v-\xi|}{|v-\xi|}\(h(v)-h(\xi)-h'(\xi)\left\langle v,\xi^\perp\right\rangle\)\frac{\delta(2-\delta)}{\delta^2+(1-\delta)|v-\xi|^2}dv}_{=:W_{3,2(z)}}.
\end{split}\end{equation*}
To analyze $W_{3,2}$ we use the pointwise estimates
\begin{eqnarray*}
\log\frac{|(1-\delta)v-\xi|}{|v-\xi|}&=&\frac12\(\log(1-\delta)+\log\(1+\frac{\delta^2}{(1-\delta)|v-\xi|^2}\)\)\\
&=&\left\{\begin{array}{ll}O\(1+\log\frac\delta{|v-\xi|}\)&|v-\xi|\le\delta,\\O\(\delta+\frac{\delta^2}{|v-\xi|^2}\)&|v-\xi|>\delta,\end{array}\right.
\end{eqnarray*}
and
\begin{eqnarray*}
\(h(v)-h(\xi)-h'(\xi)\left\langle v,\xi^\perp\right\rangle\)\frac{\delta(2-\delta)}{\delta^2+(1-\delta)|v-\xi|^2}&=&O\(\delta\frac{|v-\xi|^2}{\delta^2+(1-\delta)|v-\xi|^2}\)\\
&=&\left\{\begin{array}{ll}O\(\frac{|v-\xi|^2}\delta\)&|v-\xi|\le\delta,\\O(\delta)&|v-\xi|>\delta.\end{array}\right.
\end{eqnarray*}
From them we obtain
\begin{eqnarray*}
W_{3,2}&=&O\(\int_{|v-\xi|\le\delta}\(1+\log\frac\delta{|v-\xi|}\)\frac{|v-\xi|^2}\delta dv+\int_{|v-\xi|>\delta}\(\delta+\frac{\delta^2}{|v-\xi|^2}\)\delta dv\)\\
&=&O\(\delta^2\int_0^1\(1+\log\frac1t\)t^2dt+\delta^2\int_1^{O\(\frac1\delta\)}\(\delta+\frac1{t^2}\)dt\)\\
&=&O\(\delta^2\).
\end{eqnarray*}
To estimate $W_{3,1}$ we observe first that
\begin{eqnarray}
\label{hv}&&\(h(v)-h(\xi)-h'(\xi)\left\langle v,\xi^\perp\right\rangle\)\(\frac2{|v-\xi|^2}-\frac{2-\delta}{\delta^2+(1-\delta)|v-\xi|^2}\)\\
\nonumber&&\qquad =\;\frac{(h(v)-h(\xi)-h'(\xi)\left\langle v,\xi^\perp\right\rangle}{|v-\xi|^2}\frac{2\delta^2-\delta|v-\xi|^2}{\delta^2+(1-\delta)|v-\xi|^2}\\
\nonumber&&\qquad =\;O\(\frac\delta{\delta+|v-\xi|}\),
\end{eqnarray}
and we will divide the integral in three regions, depending on whether $v$ is closer to $\xi$ than to $f(z)$, further, or at a comparable distance. Notice that if 
$$\frac{|v-\xi|}{|v-f(z)|}\le\frac12,$$
then 
$$|v-\xi|\le2|f(z)-\xi|\quad\hbox{and }\quad\left|\log\frac{|v-\xi|}{|v-f(z)|}\right|=\log\frac{|v-f(z)|}{|v-\xi|}=O\(1+\log\frac{|f(z)-\xi|}{|v-\xi|}\).$$
On the contrary, if
$$\frac12<\frac{|v-\xi|}{|v-f(z)|}\le2,$$
then
$$|f(z)-\xi|\le3|v-f(z)|\quad\hbox{and }\quad\left|\log\frac{|v-\xi|}{|v-f(z)|}\right|=O\(\left|\frac{|v-\xi|}{|v-f(z)|}-1\right|\)=O\(\frac{|f(z)-\xi|}{|v-f(z)|}\).$$
Finally, if
$$\frac{|v-\xi|}{|v-f(z)|}>2,$$
then
$$|v-f(z)|\le|f(z)-\xi|\quad\hbox{and }\left|\log\frac{|v-\xi|}{|v-f(z)|}\right|=\log\frac{|v-\xi|}{|v-f(z)|}=O\(1+\log\frac{|f(z)-\xi|}{|v-f(z)|}\),$$
and, putting all this information together, we conclude that
\begin{eqnarray*}
W_{3,1}(z)&=&O\(\delta\int_{\frac{|v-\xi|}{|v-f(z)|}\le\frac12}\left|\log\frac{|v-\xi|}{|v-f(z)|}\right|\frac\delta{\delta+|v-\xi|}dv\right.\\
&+&\left.\delta\int_{\frac12<\frac{|v-\xi|}{|v-f(z)|}\le2}\left|\log\frac{|v-\xi|}{|v-f(z)|}\right|\frac\delta{\delta+|v-\xi|}dv+\right.\\
&+&\left.\delta\int_{\frac{|v-\xi|}{|v-f(z)|}>2}\left|\log\frac{|v-\xi|}{|v-f(z)|}\right|\frac\delta{\delta+|v-\xi|}dv\)\\
&=&O\(\delta^2\int_{|v-\xi|\le|f(z)-\xi|}\(1+\log\frac{|f(z)-\xi|}{|v-\xi|}\)\frac{dv}{\delta+|v-\xi|}\right.\\\\
&+&\left.\delta^2\int_{|v-f(z)|\ge\frac{|f(z)-\xi|}3}\frac{|f(z)-\xi|}{|v-f(z)|}\frac{2 dv}{|v-f(z)|}\right.\\
&+&\left.\delta^2\int_{|v-f(z)|\le|f(z)-\xi|}\(1+\log\frac{|f(z)-\xi|}{|v-f(z)|}\)\frac{dv}{\delta+|v-f(z)|}\)\\
&=&O\(\delta^2\int_{|u|\le|f(z)-\xi|}\(1+\log\frac{|f(z)-\xi|}{|u|}\)\frac{du}{\delta+u}+\delta^2|f(z)-\xi|\int_{\frac{|f(z)-\xi|}3}^2\frac{dt}{t^2}\)\\
&=&O\(\delta^2|f(z)-\xi|\int_0^1\(1+\log\frac1t\)\frac{dt}{\delta+|f(z)-\xi|t}+\delta^2\)\\
&=&O\(\delta|f(z)-\xi|\int_0^\frac\delta{|f(z)-\xi|}\(1+\log\frac1t\)dt+\delta^2\int_\frac\delta{|f(z)-\xi|}^1\(1+\log\frac1t\)\frac{dt}t+\delta^2\)\\
&=&O\(\delta^2\(1+\log^2\frac{|f(z)-\xi|}\delta\)\).
\end{eqnarray*}
\end{proof}

\begin{proposition}\label{e}
For any $1<p<2$,
$$\|\mathcal E\|_{L^p}=O\(\delta^\frac54+\delta|\eta|+\delta^\frac14\e+|\eta|\e+|\tau|\).$$
\end{proposition}

\begin{proof}
Recalling the definition of $h_\e$, we split the error term in three parts,
\begin{equation}\label{splitE}
\mathcal E=\underbrace{\frac1{h(\xi)}h(f(z))\(e^{W+\tau}-1\)}_{=:\mathcal E_1}+\underbrace{\frac1{h(\xi)}\e k(f(z))e^{W+\tau}}_{=:\mathcal E_2}+\underbrace{\frac1{2\pi h(\xi)}\int_{\mathbb S^1}(h(f(w))-h(\xi))dw}_{=:\mathcal E_3}.
\end{equation}
For the first one, by the inequality $\left|e^t-1\right|\le\(1+e^t\)|t|$ we obtain
\begin{eqnarray}
\nonumber\|\mathcal E_1\|_{L^p}&=&O\(\left\|e^{W+\tau}-1\right\|_{L^p}\)=O\(\left\|1+e^{W+\tau}\right\|_{L^{2p}}\left\|W+\tau\right\|_{L^{2p}}\)\\
\label{E1p}&=&O\(\delta^{1+\frac1{2p}}+\delta|\eta|+|\tau|\).
\end{eqnarray}
Since $k(1)=0$, one has $k(f(z))=O(|f(z)-\xi|+|\eta|)$, and hence, using Proposition \ref{f} and Lemma \ref{estW},
\begin{equation}\label{E2p}
\left\|\mathcal E_2\right\|_{L^p}=O\(\e\||f(z)-\xi|+|\eta|\|_{L^{2p}}\left\|e^{W+\tau}\right\|_{L^{2p}}\)=O\(\e\(\delta^\frac1{2p}+|\eta|\)\).
\end{equation}
Finally, applying Proposition \ref{intFracLap} we conclude 
$$\mathcal E_3=-\frac1{h(\xi)}\(\delta\(-\Delta\)^\frac12h(\xi)+O\(\delta^2\)\)=O\(\delta^2+\delta|\eta|\),$$
and the result follows.
\end{proof}

\medskip

\section{The projected problem}

The results contained in this section follow from \cite[Section 3 and Section 4]{bmp} by simplifying to the case $K\equiv0$. We enunciate here the statements traslated to this context.

Define
$$\mathfrak C:=\left\{\xi\in\partial\D:\;h(\xi)\ne0\right\}.$$
Notice that, since we are assuming $h(1)>0$, we will have $\xi\in\mathfrak C$ for any $\xi$ close enough to 1 (that is, for $\delta$ small enough). We consider the Hilbert space
$$\mathbf H:=\left\{\phi\in H^1(\D):\;\int\limits_{\mathbb S^1}\phi=0\right\},$$
equipped with the scalar product and the corresponding norm
$$\langle u,v\rangle:=\int\limits_\D\nabla u\nabla v\quad\hbox{and}\quad\|u\|:=\|\nabla u\|_{L^2(\D)}=\(\int\limits_\D|\nabla u|^2\)^\frac12.$$
We start stating a linear invertibility result. Consider the functions
\begin{equation}\label{generators}
\mathcal Z_1(z):=\frac{\langle z,\xi\rangle}{4h(\xi)^2},\quad\quad\quad\mathcal Z_2(z):=\frac{\left\langle z,\xi^\perp\right\rangle}{4h(\xi)^2},
\end{equation}
that satisfy
\begin{equation}\label{eqlin}
(-\Delta)^\frac12\mathcal Z_i=h(\xi)e^V\mathcal Z_i\quad\hbox{in}\;\mathbb S^1,\quad\quad\quad i=1,2.
\end{equation}
Thus, we can state the following linear invertibility result.
\begin{theorem}[see Theorem 3.3 in \cite{bmp}]
Fix $p>1$ and $\mathfrak C'\Subset\mathfrak C$. For any $\xi\in\mathfrak C'$ and $\zeta\in L^p\(\mathbb S^1\)$ such that
\begin{equation*}
\int\limits_{\mathbb S^1}\zeta=\int\limits_{\mathbb S^1}\zeta\mathcal Z_i=0,\quad i=1,2,
\end{equation*}
there exists a unique solution $\phi\in H^1(\D)$ to the problem
$$\begin{cases}(-\Delta)^\frac12\phi=h(\xi)e^{V_\xi}\phi+\zeta&\hbox{in}\;\mathbb S^1\\
\int\limits_{\mathbb S^1}\phi=\int\limits_{\mathbb S^1}\phi\mathcal Z_i=0&i=1,2.\\
\end{cases}$$
Furthermore
\begin{equation*}
\|\phi\|\le C_p\|\zeta\|_{L^p},
\end{equation*}
where the constant $C_p$ only depends on $p$ and the compact set $\mathfrak C'$.
\end{theorem}

In order to find a solution of \eqref{eqphi}, we will solve first the associated projected problem
\begin{equation}\label{projProb}
\mathcal L_0\phi=\mathcal E+\mathcal L\phi+\mathcal N(\phi)+c_0+h(\xi)e^\frac V2(c_1\mathcal Z_1+c_2\mathcal Z_2)\quad\hbox{in}\;\mathbb S^1,
\end{equation}
with $\mathcal L_0$, $\mathcal E$, $\mathcal L$, $\mathcal N$, defined in \eqref{ebdr}, $\mathcal Z_1$, $\mathcal Z_2$ given by \eqref{generators} and $c_0$, $c_1$, $c_2\in\R$. 

\begin{lemma}\label{l}Let $\phi\in\mathbf H$. Then, for any $1<p<\frac43$,
$$\left\|\mathcal L\phi\right\|_{L^p}=O\(\(\delta^{\frac54}+\delta|\eta|+\delta^\frac14\e+|\eta|\e+|\tau|\)\|\phi\|\).$$
\end{lemma}

\begin{proof}
By \eqref{ebdr} and \eqref{splitE} we can write $\mathcal L=\mathcal E_1+\mathcal E_2$ and then, 
\begin{eqnarray*}
\left\|\mathcal L\phi\right\|_{L^p}&=&O\(\left\|\mathcal E_1\phi\right\|_{L^p}+\left\|\mathcal E_2\phi\right\|_{L^p}\)\\
&=&O\(\left\|\mathcal E_1\right\|_{L^{\frac32p}}\|\phi\|_{L^{3p}}+\left\|\mathcal E_2\right\|_{L^{\frac32p}}\|\phi\|_{L^{3p}}\)\\
&=&O\(\(\delta^{\frac54}+\delta|\eta|+\delta^\frac14\e+|\eta|\e+|\tau|\)\|\phi\|\),
\end{eqnarray*}
where in the last step we have used \eqref{E1p} and \eqref{E2p}.
\end{proof}

\begin{lemma}\label{n}Let $\phi,\phi'\in\mathbf H$. For any $p>1$,
\begin{equation*}
\left\|\mathcal N(\phi)-\mathcal N(\phi')\right\|_{L^p}=O\(\|\phi-\phi'\|(\|\phi\|+\|\phi'\|)e^{O\(\|\phi\|^2+\|\phi'\|^2\)}\).
\end{equation*}
In particular
$$\|\mathcal N(\phi)\|_{L^p}=O\(\|\phi\|^2e^{O\(\|\phi\|^2\)}\).$$
\end{lemma}

\begin{proof}
Using the estimate
$$e^t-t-e^s+s=O\(|s-t|(|s|+|t|)\(1+e^{s+t}\)\),$$
Lemma \ref{estW} and the Moser-Trudinger type inequality in \cite[Lemma 3.2]{bmp} we get
\begin{eqnarray*}
&&\left\|\mathcal N(\phi)-\mathcal N(\phi')\right\|_{L^p}\\
&&\quad=\;\left\|h_\e(f(z))e^{V+W+\tau}\(e^\phi-\phi-e^{\phi'}+\phi'\)\right\|_{L^p}\\
&&\quad=\;O\(\left\|e^{W+\tau}\(|\phi-\phi'|(|\phi|+|\phi'|)\(1+e^{\phi+\phi'}\)\)\right\|_{L^p}\)\\
&&\quad=\;O\(\left\|e^{W+\tau}\right\|_{L^{4p}}\|\phi-\phi'\|_{L^{4p}}\(\|\phi\|_{L^{4p}}+\|\phi'\|_{L^{4p}}\)\left\|1+e^{\phi+\phi'}\right\|_{L^{4p}}\)\\
&&\quad=\;O\(\|\phi-\phi'\|_{L^{4p}}\(\|\phi\|_{L^{4p}}+\|\phi'\|_{L^{4p}}\)e^{O\(\|\phi+\phi'\|^2\)}\)\\
&&\quad=\;O\(\|\phi-\phi'\|(\|\phi\|+\|\phi'\|)e^{O\(\|\phi+\phi'\|^2\)}\).
\end{eqnarray*}
The second identity follows just replacing with $\phi'=0$.
\end{proof}

\begin{proposition}\label{contr}
Assume $\delta,|\eta|,|\tau|,\e\le\e_0\ll1$. Then, there exists a unique $(\phi,c_0,c_1,c_2)\in\mathbf H\times\R^3$ such that \eqref{projProb} has a solution, which additionally satisfies
\begin{equation}\label{phio}
\|\phi\|=O\(\delta^\frac54+\delta|\eta|+\delta^\frac14\e+|\eta|\e+|\tau|\).
\end{equation}
\end{proposition}

\begin{proof}
The proof follows replicating the strategy of \cite[Proposition 4.3]{bmp} so we will only highlight the differences. These come from the fact that here the error term $\mathcal{E}$ has smaller size, what allows us to perform the fixed point argument in a smaller ball. 

In particular, following their notation, by Proposition \ref{e}, Lemmas \ref{l} and \ref{n} we will have
\begin{equation}\begin{split}
&\nonumber\|\mathcal T_\xi(\phi)\|\;=\;O(\|\mathcal E\|_{L^p}+\|\mathcal L\phi\|_{L^p}+\|\mathcal N(\phi)\|_{L^p})\\
\nonumber&\quad=\;O\(\delta^\frac54+\delta|\eta|+\delta^\frac14\e+|\eta|\e+|\tau|+\(\delta^\frac54+\delta|\eta|+\delta^\frac14\e+|\eta|\e+|\tau|\)\|\phi\|+\|\phi\|^2e^{O\(\|\phi\|^2\)}\)\\
\label{tphi}&\quad=\;O\(\delta^\frac54+\delta|\eta|+\delta^\frac14\e+|\eta|\e+|\tau|+\|\phi\|^2e^{O\(\|\phi\|^2\)}\),
\end{split}\end{equation}
and
\begin{equation*}\begin{split}
&\left\|\mathcal T_\xi(\phi)-\mathcal T_\xi(\phi')\right\| =O\(\left\|\mathcal L(\phi-\phi')\right\|_{L^p}+\left\|\mathcal N(\phi)-\mathcal N(\phi')\right\|_{L^p}\)\\
&\qquad=\;O\(\|\phi-\phi'\|\(\delta^\frac54+\delta|\eta|+\delta^\frac14\e+|\eta|\e+|\tau|+(\|\phi\|+\|\phi'\|)e^{O\(\|\phi\|^2+\|\phi'\|^2\)}\)\).
\end{split}\end{equation*}
Choosing $R$ large enough, from \eqref{tphi} we have 
$$\|\phi\|\le R\(\delta^\frac54+\delta|\eta|+\delta^\frac14\e+|\eta|\e+|\tau|\)\quad\Rightarrow\quad\|\mathcal T_\xi(\phi)\|\le R\(\delta^\frac54+\delta|\eta|+\delta^\frac14\e+|\eta|\e+|\tau|\);$$
proceeding as in \cite[Proposition 4.3]{bmp} we conclude that $\mathcal T_\xi$ is a contraction on a suitable ball and it has a unique fixed point that satisfies \eqref{phio}.
\end{proof}

\medskip

\section{Estimates on the projections}

Let $\phi$ be the solution to the problem \eqref{projProb} provided by Proposition \ref{contr}. Thus, if we prove
$$c_0=c_1=c_2=0,$$
then $\phi$ will be a solution of \eqref{eqphi}. The goal of this section will be to identify the exact expression of these constants.

We begin multiplying \eqref{projProb} by $\mathcal Z_1$ and integrating. Since
$$\int\limits_{\mathbb S^1}h(\xi)e^\frac V2\mathcal Z_1\mathcal Z_2=\int\limits_{\mathbb S^1}h(\xi)e^\frac V2\mathcal Z_i=0\quad i=1,2,$$
we deduce
$$c_1\int\limits_{\mathbb S^1}h(\xi)e^\frac V2\mathcal Z_1^2=\int\limits_{\mathbb S^1}\mathcal L_0\phi\mathcal Z_1-\int\limits_{\mathbb S^1}\mathcal E\mathcal Z_1-\int\limits_{\mathbb S^1}\mathcal L\phi\mathcal Z_1-\int\limits_{\mathbb S^1}\mathcal N(\phi)\mathcal Z_1.$$
Integrating by parts and using \eqref{eqlin},
$$\int\limits_{\mathbb S^1}\(\mathcal L_0\phi\)\mathcal Z_1=\int\limits_{\mathbb S^1}\(\mathcal L_0\mathcal Z_1\)\phi=0,$$
and hence
\begin{equation}
c_1\int\limits_{\mathbb S^1}h(\xi)e^\frac V2\mathcal Z_1^2=-\int\limits_{\mathbb S^1}\mathcal E\mathcal Z_1-\int\limits_{\mathbb S^1}\mathcal L\phi\mathcal Z_1-\int\limits_{\mathbb S^1}\mathcal N(\phi)\mathcal Z_1.\label{c1}
\end{equation}
We proceed analogously with $\mathcal Z_2$ to obtain
\begin{equation}
c_2\int\limits_{\mathbb S^1}h(\xi)e^\frac V2\mathcal Z_2^2=-\int\limits_{\mathbb S^1}\mathcal E\mathcal Z_2-\int\limits_{\mathbb S^1}\mathcal L\phi\mathcal Z_2-\int\limits_{\mathbb S^1}\mathcal N(\phi)\mathcal Z_2.\label{c2}
\end{equation}
Notice that, since $\phi\in\mathbf H$,
$$\int\limits_{\mathbb S^1}\mathcal L_0\phi=0,$$
and integrating \eqref{projProb} gives
\begin{equation}2\pi c_0=-\int\limits_{\mathbb S^1}\mathcal E-\int\limits_{\mathbb S^1}\mathcal L\phi-\int\limits_{\mathbb S^1}\mathcal N(\phi).\label{c0}\end{equation}
Let us compute the terms involved.
\begin{proposition}\label{E1}
\begin{eqnarray*}\int_{\mathbb S^1}\mathcal E\mathcal Z_1&=&\(\frac\pi2\frac{h''(1)}{h(1)^3}-\frac{Q(h)}{\pi h(1)^4}\)\delta^2+\frac\pi2\frac{(-\Delta)^\frac12h'(1)}{h(1)^3}\delta\eta+\frac\pi2\frac{(-\Delta)^\frac12k(1)}{h(1)^3}\delta\e\\
&+&O\(\delta^{2+\alpha}+\delta\e(\delta+|\eta|)+\delta|\eta|^2+|\tau|^2\),
\end{eqnarray*}
with
$$Q(h):=\int_{\mathbb S^1\times\mathbb S^1}\log\frac1{|z-w|}\frac{h(z)-h(1)}{|z-1|^2}\frac{h(w)-h(1)}{|w-1|^2}dwdz.$$
\end{proposition}

\begin{proof}
Proceeding as in \cite[Proposition 5.2]{bmp} we obtain
\begin{eqnarray}
\nonumber\int_{\mathbb S^1}\mathcal E\mathcal Z_1&=&\frac1{h(1)}(1+O(|\eta|+|\tau|))\int_{\mathbb S^1}(h(f(z))-h(\xi))\mathcal Z_1\\
\nonumber&+&\frac1{2h(1)}(1+O(|\eta|))\int_{\mathbb S^1}(h(f(z))-h(\xi))W\mathcal Z_1+\int_{\mathbb S^1}O\((|W|+|\tau|)^2\(1+e^{W+\tau}\)\)\\
\label{E1int}&+&\frac\pi2\frac{(-\Delta)^\frac12k(1)}{h(1)^3}\delta\e+O(\delta\e(\delta+|\eta|)).
\end{eqnarray}
To estimate the first term we notice that, due to the assumption $h\in C^{2,\alpha}\(\mathbb S^1\)$, the ratio
$$\frac{h(y)-h(\xi)-h'(\xi)\left\langle y,\xi^\perp\right\rangle}{|y-\xi|^2}$$
will be of class $C^{0,\alpha}$ on the whole $\mathbb S^1\times\mathbb S^1$, hence
\begin{eqnarray}
\nonumber\frac{h(y)-h(\xi)-h'(\xi)\left\langle y,\xi^\perp\right\rangle}{|y-\xi|^2}&=&\frac{h(y)-h(1)-h'(1)\left\langle y,\imath\right\rangle}{|y-1|^2}+O(|\xi-1|^\alpha)\\
\label{holder}&=&\frac{h(y)-h(1)}{|y-1|^2}+O(|\eta|^\alpha)\\
\nonumber&=&\frac{h''(1)}2+O(|y-1|^\alpha+|\eta|^\alpha)
\end{eqnarray}
Therefore, by making the change of variable $y=f(z)$ we get
\begin{eqnarray}
\nonumber&&\int_{\mathbb S^1}(h(f(z))-h(\xi))\langle z,\xi\rangle dz\\
\nonumber&=&\int_{\mathbb S^1}(h(y)-h(\xi))\frac{\delta(2-\delta)}{\delta^2+(1-\delta)|y-\xi|^2}\(-1+\frac{\delta^2(1+\langle w,\xi\rangle)}{\delta^2+(1-\delta)|y-\xi|^2}\)dy\\
\nonumber&=&2\delta(1+O(\delta))\int_{\mathbb S^1}\frac{(h(y)-h(\xi)-h'(\xi)\left\langle y,\xi^\perp\right\rangle}{\delta^2+(1-\delta)|y-\xi|^2}\(-1+\frac{\delta^2(1+\langle w,\xi\rangle)}{\delta^2+(1-\delta)|y-\xi|^2}\)dy\\
\nonumber&=&2\delta(1+O(\delta))\(-\int_{\mathbb S^1}\frac{h(y)-h(\xi)-h'(\xi)\left\langle y,\xi^\perp\right\rangle}{(1-\delta)|y-\xi|^2}dy\right.\\
\nonumber&+&\left.\int_{\mathbb S^1}\(\frac{\(h(y)-h(\xi)-h'(\xi)\left\langle y,\xi^\perp\right\rangle\)\delta^2}{(1-\delta)|y-\xi|^2\(\delta^2+(1-\delta)|y-\xi|^2\)}+\frac{\(h(y)-h(\xi)-h'(\xi)\left\langle y,\xi^\perp\right\rangle\)\delta^2(1+\langle y,\xi\rangle)}{\(\delta^2+(1-\delta)|y-\xi|^2\)^2}\)dy\)\\
\nonumber&=&2\delta(1+O(\delta))\(\pi(-\Delta)^\frac12h(\xi)(1+O(\delta))\right.\\
\nonumber&+&\left.\delta^2\int_{\mathbb S^1}\frac{h(y)-h(\xi)-h'(\xi)\left\langle y,\xi^\perp\right\rangle}{|y-\xi|^2}\(\frac1{\delta^2+(1-\delta)|y-\xi|^2}+\frac{|y-\xi|^2(1+\langle y,\xi\rangle)}{\(\delta^2+(1-\delta)|y-\xi|^2\)^2}\)dy\)\\
\nonumber&=&2\delta(1+O(\delta))\(\pi(-\Delta)^\frac12h'(1)\eta+O((\delta+|\eta|)|\eta|)\right.\\
\nonumber&+&\left.\delta^2\int_{\mathbb S^1}\(\frac{h''(1)}2+O(|y-1|^\alpha+|\eta|^\alpha)\)\frac{\delta^2+(3-\delta)|y-\xi|^2}{\(\delta^2+(1-\delta)|y-\xi|^2\)^2}dy\)\\
\nonumber&=&2\delta(1+O(\delta))\(\pi(-\Delta)^\frac12h'(1)\eta+O((\delta+|\eta|)|\eta|)\right.\\
\nonumber&+&\left.\delta^2\(\frac{h''(1)}2+O(|\eta|^\alpha)\)\int_{\mathbb S^1}\(\frac{\delta^2+3|y-\xi|^2}{\(\delta^2+|y-\xi|^2\)^2}(1+O(\delta))+O\(\frac{|y-1|^\alpha}{\delta^2+|y-\xi|^2}\)\)dy\)\\
\nonumber&=&2\delta(1+O(\delta))\(\pi(-\Delta)^\frac12h'(1)\eta+O((\delta+|\eta|)|\eta|)\right.\\
\nonumber&+&\left.\delta\(\frac{h''(1)}2+O(|\eta|^\alpha)\)\int_{t=O\(\frac1\delta\)}\(\frac{1+3t^2}{\(1+t^2\)^2}(1+O(\delta))+O\(\frac{\delta^\alpha t^\alpha+|\eta|^\alpha}{1+t^2}\)\)dt\)\\
\nonumber&=&2\delta(1+O(\delta))\(\pi\eta(-\Delta)^\frac12h'(1)+O((\delta+|\eta|)|\eta|)+\delta\(\frac{h''(1)}2+O(|\eta|^\alpha)\)(2\pi+O(\delta^\alpha+|\eta|^\alpha))\)\\
\label{E1int1}&=&2\pi(-\Delta)^\frac12h'(1)\delta\eta+2\pi h''(1)\delta^2+O\(\delta\(\delta^{1+\alpha}+|\eta|^2\)\).
\end{eqnarray}
To see the second term we use the expression of $W$ given by \eqref{wz}. Notice first that, by Proposition \ref{f}, the lower term can be estimated as
\begin{eqnarray*}
\int_{\mathbb S^1}(h(f(z))-h(\xi))&=&O\(\delta\log\frac1\delta(\delta+|\eta|)\(1+\log^2\frac1{\delta+|z+\xi|}\)\)\mathcal Z_1\\
&=&O\(\delta\log\frac1\delta(\delta+|\eta|)\int_{\mathbb S^1}\frac\delta{\delta+|z+\xi|}\(1+\log^2\frac1{\delta+|z+\xi|}\)dz\)\\
&=&O\(\delta\log\frac1\delta(\delta+|\eta|)\(\int_{|z+\xi|\le\delta}\(1+\log^2\frac1\delta\)dz+\right.\right.\\
&+&\left.\left.\int_{|z+\xi|>\delta}\frac\delta{|z+\xi|}\(1+\log^2\frac1{|z+\xi|}\)dz\)\)\\
&=&O\(\delta^2\log^4\frac1\delta(\delta+|\eta|)\).
\end{eqnarray*}
On the other hand, we can approximate $\langle z,\xi\rangle$ with $-1$ since
\begin{eqnarray*}
\int_{\mathbb S^1}(h(f(z))-h(\xi))W(z)(\langle z,\xi\rangle+1)dz&=&O\(\|W\|_{L^2}\left\|\frac\delta{\delta+|z+\xi|}|z+\xi|^2\right\|_{L^2}\)\\
&=&O\(\delta\|W\|_{L^2}\)\\
&=&O\(\delta^\frac52+\delta^2|\eta|\);
\end{eqnarray*}
therefore, using again \eqref{holder}, the main order term will be given by
\begin{eqnarray*}
&&\int_{\mathbb S^1}(h(f(z))-h(\xi))\(\frac4{\pi h(\xi)}\delta\int_{\mathbb S^1}\log\frac{|v-\xi|}{|v-f(z)|}\frac{h(v)-h(\xi)-h'(\xi)\left\langle v,\xi^\perp\right\rangle}{|v-\xi|^2}dv\)dz\\
&=&\frac4{\pi h(\xi)}\delta\int_{\mathbb S^1\times\mathbb S^1}\(h(y)-h(\xi)-h'(\xi)\left\langle y,\xi^\perp\right\rangle+O(|\eta||y-\xi|)\)\frac{\delta(2-\delta)}{\delta^2+(1-\delta)|y-\xi|^2}\\
&&\cdot\log\frac{|v-\xi|}{|v-y|}\frac{h(v)-h(\xi)-h'(\xi)\left\langle v,\xi^\perp\right\rangle}{|v-\xi|^2}dvdy\\
&=&\frac8{\pi h(\xi)}\delta^2\int_{\mathbb S^1\times\mathbb S^1}\log\frac{|v-\xi|}{|v-y|}\(\frac{h(y)-h(\xi)-h'(\xi)\left\langle y,\xi^\perp\right\rangle)}{|y-\xi|^2}+O\(\frac{|\eta||y-\xi|}{\delta^2+|y-\xi|^2}\)\right.\\
&+&\left.O\(\frac\delta{\delta+|y+\xi|}\)\)\cdot\frac{h(v)-h(\xi)-h'(\xi)\left\langle v,\xi^\perp\right\rangle}{|v-\xi|^2}dvdy\\
&=&\frac8{\pi h(1)}(1+O(|\eta|))\delta^2\int_{\mathbb S^1\times\mathbb S^1}\log\frac{|v-\xi|}{|v-y|}\frac{h(y)-h(1)}{|y-1|^2}\frac{h(v)-h(1)}{|v-1|^2}(1+O(|\eta|^\alpha))^2dvdy\\
&+&O\(\delta^2\int_{\mathbb S^1\times\mathbb S^1}\left|\log\frac{|v-\xi|}{|v-y|}\right|\(|\eta|+\frac\delta{\delta+|y+\xi|}+\frac{|\eta||y-\xi|}{\delta^2+|y-\xi|^2}\)\)\\
&=&\frac8{\pi h(1)}(1+O(|\eta|^\alpha))\delta^2\int_{\mathbb S^1\times\mathbb S^1}\log\frac{|v-\xi|}{|v-y|}\frac{h(v)-h(1)}{|v-1|^2}\frac{h(y)-h(1)}{|y-1|^2}dvdy\\
&+&O\(\delta^2\int_{\mathbb S^1\times\mathbb S^1}\left|\log\frac{|v-\xi|}{|v-y|}\right|\(|\eta|+\frac\delta{\delta+|y+\xi|}+\frac{|\eta||y-\xi|}{\delta^2+|y-\xi|^2}\)\)\\
&=&\frac8{\pi h(1)}(1+O(|\eta|^\alpha))\delta^2\int_{\mathbb S^1\times\mathbb S^1}\log\frac1{|v-y|}\frac{h(v)-h(1)}{|v-1|^2}\frac{h(y)-h(1)}{|y-1|^2}dvdy\\
&+&\frac8{\pi h(1)}(1+O(|\eta|^\alpha))\delta^2\int_{\mathbb S^1}\log|v-\xi|\frac{h(v)-h(1)}{|v-1|^2}dv\int_{\mathbb S^1}\frac{h(y)-h(1)}{|y-1|^2}dy\\
&+&O\(\delta^2\int_{\mathbb S^1\times\mathbb S^1}\(2+\log\frac1{|v-\xi|}+\log\frac1{|v-y|}\)\(|\eta|+\frac\delta{\delta+|y+\xi|}+\frac{|\eta||y-\xi|}{\delta^2+|y-\xi|^2}\)\)\\
&=&\frac8{\pi h(1)}(1+O(|\eta|^\alpha))\delta^2\int_{\mathbb S^1\times\mathbb S^1}\log\frac1{|v-y|}\frac{h(v)-h(1)}{|v-1|^2}\frac{h(y)-h(1)}{|y-1|^2}dvdy\\
&+&O\(\delta^2\(|\eta|+|\eta|\log\frac1\delta+\int_{\mathbb S^1\times\mathbb S^1}\(1+\log\frac1{|u-\xi|}\)\frac\delta{\delta+|y+\xi|}\)dudy\)\\
&=&\frac8{\pi h(1)}\delta^2\int_{\mathbb S^1\times\mathbb S^1}\log\frac1{|v-y|}\frac{h(v)-h(1)}{|v-1|^2}\frac{h(y)-h(1)}{|y-1|^2}dvdy\\
&+&O\(\delta^2|\eta|^\alpha+\delta^2|\eta|\log\frac1\delta+\delta^3\log\frac1\delta\), 
\end{eqnarray*}
where we have used estimate \eqref{hv} and the fact that $h'(1)=(-\Delta)^\frac12h(1)=0$.
Finally, by Lemma \ref{estW},
\begin{eqnarray}
\nonumber\int_{\mathbb S^1}O\((|W|+|\tau|)^2\(1+e^{W+\tau}\)\)&=&O\(\left\|(|W|+|\tau|)^2\right\|_{L^2}\left\|1+e^{W+\tau}\right\|_{L^2}\)\\
\nonumber&=&O\(\|W\|_{L^4}^2+|\tau|^2\)\\
\label{E1int5}&=&O\(\delta^\frac52+\delta^2|\eta|^2+|\tau|^2\).
\end{eqnarray}
Replacing \eqref{E1int1}-\eqref{E1int5} into \eqref{E1int} the result follows.
\end{proof}

\begin{proposition}\label{E2}
\begin{equation*}
\int_{\mathbb S^1}\mathcal E\mathcal Z_2=-\frac\pi2\frac{(-\Delta)^\frac12h'(1)}{h(1)^3}\delta^2+\frac\pi2\frac{h''(1)}{h(1)^3}\delta\eta+\frac\pi2\frac{k'(1)}{h(1)^3}\delta\e+O\(\delta^\frac52+\delta\e(\delta+|\eta|)+\delta|\eta|^{1+\alpha}+|\tau|^2\).
\end{equation*}

\end{proposition}

\begin{proof}
Following \cite[Proposition 5.4]{bmp} we can write the integral as
\begin{eqnarray}
\nonumber\int_{\mathbb S^1}\mathcal E\mathcal Z_2&=&\frac1{h(1)}(1+O(|\eta|+|\tau|))\int_{\mathbb S^1}(h(f(z))-h(\xi))\mathcal Z_2\\
\nonumber&+&\frac1{2h(1)}(1+O(|\eta|))\int_{\mathbb S^1}(h(f(z))-h(\xi))W\mathcal Z_2+\frac\pi2\frac{k'(1)}{h(1)^3}\delta\e\\
\label{EZ2}&+&O\(\delta^\frac52+(\delta+|\eta|)\delta\e+|\tau|^2\).
\end{eqnarray}
To estimate the first term we make the change of variable $y=f(z)$ and we integrate by parts to obtain
\begin{eqnarray*}
&&\int_{\mathbb S^1}(h(f(z))-h(\xi))\left\langle z,\xi^\perp\right\rangle dz\\
&=&\int_{\mathbb S^1}(h(y)-h(\xi))\frac{\delta(2-\delta)}{\delta^2+(1-\delta)|y-\xi|^2}\frac{\delta(2-\delta)\left\langle y,\xi^\perp\right\rangle}{\delta^2+(1-\delta)|y-\xi|^2}dy\\
&=&4\delta^2(1+O(\delta))\int_{\mathbb S^1}(h(y)-h(\xi))\(-\frac1{2(1-\delta)\(\delta^2+(1-\delta)|y-\xi|^2\)}\)'dy\\
&=&-4\delta^2(1+O(\delta))\int_{\mathbb S^1}h'(y)\(-\frac1{2(1-\delta)\(\delta^2+(1-\delta)|y-\xi|^2\)}\)dy\\
&=&2\delta^2(1+O(\delta))\(\int_{\mathbb S^1}\frac{h'(y)-h'(\xi)}{\delta^2+(1-\delta)|y-\xi|^2}dy+h'(\xi)\int_{\mathbb S^1}\frac{dy}{\(\delta^2+(1-\delta)|y-\xi|^2\)}\)\\
&=&2\delta^2(1+O(\delta))\(-\pi(-\Delta)^\frac12h'(\xi)+O(\delta)+h'(\xi)\int_{\mathbb S^1}\frac{dy}{\(\delta^2+|y-\xi|^2\)}(1+O(\delta))\)\\
&=&2\delta^2(1+O(\delta))\(-\pi(-\Delta)^\frac12h'(1)+O(\delta+|\eta|^\alpha)\right.\\
&& \left.+\(\eta h''(1)+O\(|\eta|^{1+\alpha}\)\)\int_{\mathbb S^1}\frac{dy}{\(\delta^2+|y-\xi|^2\)}(1+O(\delta))\)\\
&=&2\delta^2(1+O(\delta))\(-\pi(-\Delta)^\frac12h'(1)+O(\delta+|\eta|^\alpha)\right.\\
&&\left.+\(\eta h''(1)+O\(|\eta|^{1+\alpha}\)\)\frac1\delta\int_{t=O\(\frac1\delta\)}\frac{dt}{1+t^2}(1+O(\delta))\)\\
&=&2\delta^2(1+O(\delta))\(-\pi(-\Delta)^\frac12h'(1)+O(\delta+|\eta|^\alpha)+\(\eta h''(1)+O\(|\eta|^{1+\alpha}\)\)\(\frac\pi\delta+O(1)\)\)\\
&=&-2\pi(-\Delta)^\frac12h'(1)\delta^2+2\pi h''(1)\delta\eta+O\(\delta\(\delta^2+|\eta|^{1+\alpha}\)\),
\end{eqnarray*}
where we applied that $h\in C^{2,\alpha}\(\mathbb S^1\)$, and hence $(-\Delta)^\frac12h\in C^{1,\alpha}\(\mathbb S^1\)$.

Using Lemma \ref{estW} and Proposition \ref{f} we can estimate the second integral as
\begin{eqnarray*}
\int_{\mathbb S^1}(h(f(z))-h(\xi))W(z)\left\langle z,\xi^\perp\right\rangle dz&=&O\(\|W\|_{L^2}\left\|\frac\delta{\delta+|z+\xi|}|z+\xi|\right\|_{L^2}\)\\
&=&O\(\delta^\frac52+\delta^2|\eta|\),
\end{eqnarray*}
where we have used the fact $\left\langle z,\xi^\perp\right\rangle=O(|z+\xi|)$. 

Replacing in \eqref{EZ2} we conclude.
\end{proof}

\begin{proposition}\label{E0}
\begin{equation*}\int_{\mathbb S^1}\mathcal E=\tau\(\pi h(1)+O(|\eta|+|\tau|)\)+O\(\delta^2\log^2\frac1\delta+\delta|\eta|+\delta\log\frac1\delta\e\).
\end{equation*}
\end{proposition}

\begin{proof}
The result follows from \cite[Proposition 5.6]{bmp} by fixing $K\equiv0$.
\end{proof}

We can finally write the exact expressions of the constants $c_0,c_1$ and $c_2$.

\begin{corollary}\label{mainorder}
The constants $c_0,c_1,c_2$ in problem \eqref{projProb} satisfy:
\begin{eqnarray*}
c_1&=&-16 h(1)\delta\bigg(\delta\underbrace{\(h''(1)-\frac{2Q(h)}{\pi^2h(1)}\)}_{=:\mathfrak a_{11}}+\eta\underbrace{(-\Delta)^\frac12h'(1)}_{=:\mathfrak a_{12}}+\e\underbrace{(-\Delta)^\frac12k(1)}_{=:\mathfrak b_1}\bigg)\\
&+&O\(\delta^{2+\alpha}+\delta\e(\delta+|\eta|)+\delta|\eta|^2+|\tau|^2\),\\
c_2&=&-16 h(1)\delta\(\delta\underbrace{\(-(-\Delta)^\frac12h'(1)\)}_{=:\mathfrak a_{21}}+\eta\underbrace{h''(1)}_{=:\mathfrak a_{22}}+\e\underbrace{k'(1)}_{=:\mathfrak b_2}\)\\
&+&O\(\delta^\frac52+\delta\e(\delta+|\eta|)+\delta|\eta|^{1+\alpha}+|\tau|^2\),\\
c_0&=&-h(1)\tau+O\(|\eta||\tau|+|\tau|^2+\delta^2\log^2\frac1\delta+\delta|\eta|+\delta\log\frac1\delta\e\)
\end{eqnarray*}
where $Q$ is defined in \eqref{q}.
\end{corollary}

\begin{proof}
By \cite[Proposition 7.9]{bmp} with $K\equiv0$ we know that
\begin{equation}\label{c11}
\int_{\mathbb S^1}h(\xi)e^{\frac V2}\mathcal Z_1^2=\frac\pi{16h(1)^4}(1+O(|\eta|).
\end{equation}
Furthermore, by Lemmas \ref{l} and \ref{n}, for $1<p<\frac43$,
\begin{eqnarray}
\nonumber\int_{\mathbb S^1}\mathcal L\phi\mathcal Z_1+\int_{\mathbb S^1}\mathcal N(\phi)\mathcal Z_1&=&O\(\|\mathcal L\phi\|_{L^p}+\|\mathcal N(\phi)\|_{L^p}\)\\
\nonumber&=&O\(\(\delta^{\frac54}+\delta|\eta|+\delta^\frac14\e+|\eta|\e+|\tau|\)\|\phi\|+\|\phi\|^2e^{O\(\|\phi\|^2\)}\)\\
\label{c12}&=&O\(\delta^{\frac52}+\delta^2|\eta|^2+\delta^\frac12\e^2+|\eta|^2\e^2+|\tau|^2\),
\end{eqnarray}
where in the last step we have used Proposition \ref{contr}. Applying \eqref{c11}, \eqref{c12} and Proposition \ref{E1} in \eqref{c1} we obtain the expression for $c_1$. 

The identities for $c_2$ and $c_0$ similarly follow from \eqref{c2} and \eqref{c0} using Propositions \ref{E2} and \ref{E0} respectively.
\end{proof}

\medskip

\section{The finite dimensional reduction: proof of Theorem \ref{main}}

Let $\delta,|\eta|,|\tau|,\e$ be small enough so that Proposition \ref{contr} can be applied to find a solution to \eqref{projProb}. If we choose $\delta,\eta,\tau$ in such a way that $c_0=c_1=c_2=0$ then $\phi$ also solves \eqref{eqphi}, and hence we get a solution to the problem \eqref{p}. 

We study first the cases of $c_1$ and $c_2$. From Corollary \ref{mainorder} and the estimates on $\|\phi\|$,
assuming that
$$\eta=s\e\;\quad\hbox{and}\;\quad\delta=d\e\;\;\hbox{with}\;d>0\;\hbox{if}\;\e>0\;\hbox{or}\;d<0\;\hbox{if}\;\e<0,$$
we have that $c_1=c_2=0$ if
$$\left\{\begin{array}{ll}
\mathfrak a_{11}d+\mathfrak a_{12}s+\mathfrak b_1+o_\e(1)=0,\\
\mathfrak a_{21}d+\mathfrak a_{22}s+\mathfrak b_2+o_\e(1)=0.
\end{array}\right.$$
This system can be rewritten as $\mathfrak F_\e(d,s)=\mathfrak F_0(d,s)+o_\e(1)=0$ where $\mathfrak F_0:\R^2\to\R^2$ is defined by
$$\mathfrak F_0(d,s):=\mathfrak A\(\begin{matrix}d\\s\\\end{matrix}\)+\mathfrak B,\;\hbox{with}\;\mathfrak A:=\(\begin{matrix}\mathfrak a_{11}&\mathfrak a_{12}\\\mathfrak a_{21}&\mathfrak a_{22}\\\end{matrix}\)\;\hbox{and}\;\mathfrak B:=\(\begin{matrix}\mathfrak b_1\\\mathfrak b_2\\\end{matrix}\).$$
Therefore, if 
$$\det\mathfrak A=\det\(\begin{matrix}\mathfrak a_{11}&\mathfrak a_{12}\\\mathfrak a_{21}&\mathfrak a_{22}\\\end{matrix}\)\ne0,\quad\hbox{i.e.,}\quad\mathfrak a_{11}\mathfrak a_{22}-\mathfrak a_{21}\mathfrak a_{12}\ne0\;\hbox{(see \eqref{nondeg})},$$ 
there exists a unique $(d_0,s_0)\in\R^2$ such that $\mathfrak F_0(d_0,s_0)=0$ with $d_0\ne0$ if
$$\det\(\begin{matrix}\mathfrak b_1&\mathfrak a_{12}\\\mathfrak b_2&\mathfrak a_{22}\\\end{matrix}\)\ne0,\quad\hbox{i.e.,}\quad\mathfrak a_{22}\mathfrak b_1-\mathfrak a_{12}\mathfrak b_2\ne0\;\hbox{(see \eqref{cond})}.$$ 
Moreover, the Brouwer degree of $\mathfrak F_\e$ is not zero and since $\mathfrak F_\e\to\mathfrak F_0$ uniformly on compact sets of $\R\times\R\setminus\{0\}$, there exists $(d_\e,s_\e)\in\R\times\R\setminus\{0\}$ such that $\mathfrak F_\e(d_\e,s_\e)=0$ with $(d_\e,s_\e)\to(d_0,s_0)$ as $\e\to0$. 

Once $d_\e, s_\e$ are fixed the existence of a $\tau=\tau_\e$ so that
$$-h(1)\tau+O\(|\eta||\tau|+|\tau|^2+\delta^2\log^2\frac1\delta+\delta|\eta|+\delta\log\frac1\delta\e\)=0$$
is immediate, and we conclude that $c_0=0$. Notice that $\tau_\e=o_\e(\e)$. This finishes the proof of the existence of solution.

Thanks to the estimates on $\mathcal E,\mathcal L,\mathcal N$ and $\|\phi\|$, from \eqref{eqphi} we get $\left\|(-\Delta)^\frac12\phi\right\|_{L^p}=o_\e(1)$ for some $p>1$, and therefore $\|\phi\|_{L^\infty}=o_\e(1)$. Moreover, since $V\(f^{-1}(z)\)$ and $W\(f^{-1}(z)\)$ both concentrate at $\xi=1$, we conclude that the solution
$$u=V\(f^{-1}(z)\)+W\(f^{-1}(z)\)+\tau+\phi\(f^{-1}(z)\)$$
concentrates at $\xi=1$. This ends the proof of Theorem \ref{main}.


\begin{thebibliography}{10}

\bibitem{bmp}
L.~Battaglia, M.~Medina, and A.~Pistoia.
\newblock Large conformal metrics with prescribed gaussian and geodesic
curvatures.
\newblock {\em  Calc. Var. Partial Differential Equations} 60 (2021), no. 1, Paper No. 39, 47 pp. 


\bibitem{cl}
K.~C. Chang and J.~Q. Liu.
\newblock A prescribing geodesic curvature problem.
\newblock {\em Math. Z.}, 223(2):343--365, 1996.


\bibitem{dmr}
F.~Da~Lio, L.~Martinazzi, and T.~Rivi\`ere.
\newblock Blow-up analysis of a nonlocal {L}iouville-type equation.
\newblock {\em Anal. PDE}, 8(7):1757--1805, 2015.

\bibitem{gl}
Y.-X. Guo and J.-Q. Liu.
\newblock Blow-up analysis for solutions of the {L}aplacian equation with
exponential {N}eumann boundary condition in dimension two.
\newblock {\em Commun. Contemp. Math.}, 8(6):737--761, 2006.

\bibitem{jlmr}
A.~Jevnikar, R.~L\'{o}pez-Soriano, M.~Medina, and D.~Ruiz.
\newblock Blow-up analysis of conformal metrics of the disk with prescribed
gaussian and geodesic curvatures.
\newblock {\em https://arxiv.org/pdf/2004.14680.pdf}, 2020.

\bibitem{lh}
P.~Liu and W.~Huang.
\newblock On prescribing geodesic curvature on {$D^2$}.
\newblock {\em Nonlinear Anal.}, 60(3):465--473, 2005.

\bibitem{ou}
B.~Ou.
\newblock A uniqueness theorem for harmonic functions on the upper-half plane.
\newblock {\em Conform. Geom. Dyn.}, 4:120--125, 2000.

\bibitem{z}
H.~Zhang.
\newblock Prescribing the boundary geodesic curvature on a compact scalar-flat
{R}iemann surface via a flow method.
\newblock {\em Pacific J. Math.}, 273(2):307--330, 2015.

\bibitem{zhang}
L.~Zhang.
\newblock Classification of conformal metrics on {${\bf R}^2_+$} with constant
{G}auss curvature and geodesic curvature on the boundary under various
integral finiteness assumptions.
\newblock {\em Calc. Var. Partial Differential Equations}, 16(4):405--430,
2003.

\end{thebibliography}
\end{document}